\documentclass{amsart}

\makeatletter
\@namedef{subjclassname@2020}{%
  \textup{2020} Mathematics Subject Classification}
\makeatother
\newtheorem{theorem}{Theorem}[section]
\usepackage{esint}
\theoremstyle{definition}

\newtheorem{corollary}[theorem]{Corollary}

\newtheorem{lemma}[theorem]{Lemma}
\theoremstyle{remark}

\numberwithin{equation}{section}

\begin{document}

\title[ Harnack inequality  under integral Ricci curvature bounds]{Harnack inequality for   nonlinear parabolic equations under integral Ricci curvature bounds }

\author{Shahroud Azami}
\address{Department of  Pure Mathematics, Faculty of Science,
Imam Khomeini International University,
Qazvin, Iran. \\
              Tel.: +98-28-33901321\\
              Fax: +98-28-33780083\\
             }

\email{azami@sci.ikiu.ac.ir}
%
%




\begin{abstract}
Let $(M^{n},g)$  be a complete  Riemannian manifold.  In this paper, we establish   a  space-time  gradient estimates   for positive solutions of   nonlinear parabolic  equations
$$\partial_{t}u(x,t)=\Delta u(x,t)+a u(x,t)(\log u(x,t))^{b}+q(x,t)A(u(x,t)),$$
on geodesic balls $B(O,r)$ in $M$ with $0<r\leq r$ for $p>\frac{n}{2}$ when integral Ricci curvature $k(p,1)$ is small enough.
 By integrating the gradient estimates, we find the corresponding Harnack inequalities.
\end{abstract}

\maketitle

\section{Introduction}
For $x\in M$ let $\rho(x)$ be the smallest eigenvalue for the Ricci tensor  $Ric:T_{x}M\to T_{x}M$, we define $Ric_{-}(x)=\max\{0,-\rho(x)\}$. Let $B(x,r)$ is denotes the geodesic ball with radius $r$ in $M$ centered at $x$. Then the following \cite{DWZ}, for any  constants $p,r>0$, consider $$k(x,p,r)=r^{2}\left(\fint_{B(x,r)}|Ric_{-}|^{p}d\mu\right)^{\frac{1}{p}},\qquad k(p,r)=\mathop{\sup}\limits_{x\in M}k(x,p,r).$$

The gradient estimates are important tools in geometric analysis. They  were originated by Yau \cite{YA} and Li-Yau \cite{LY}. Li and Yau \cite{LY}, established parabolic gradient estimates on solutions $u$ to the heat equation $u_{t}=\Delta u$ on complete Riemannian manifold $(M^{n},g)$ with a fixed metric and Ricci curvature bounded from below by $-K$, where $K\geq0$, in fact they  proved
\begin{equation}\label{ly}
\frac{|\nabla u|^{2}}{u^{2}}-\alpha \frac{u_{t}}{u}\leq \frac{n\alpha^{2}K}{2(\alpha-1)}+\frac{n\alpha^{2}}{2t}
\end{equation}
where $\alpha>1$ is a constant.
This estimate  called the Li-Yau estimate.
In  \cite{RH}, Hamilton generalized the constant $\alpha$  in inequality (\ref{ly})  to the function $\alpha(t)=e^{2Kt}$, and obtained the following inequality
\begin{equation*}
\frac{|\nabla u|^{2}}{u^{2}}-e^{2Kt} \frac{u_{t}}{u}\leq e^{4Kt}\frac{n}{2t}.
\end{equation*}
This type gradient estimate is called Hamilton type estimate. In 2006,  Sun \cite{SU} obtained another type estimate for positive solution to the equation $u_{t}=\Delta u$ as follows
\begin{equation*}
\frac{|\nabla u|^{2}}{u^{2}}-(1+\frac{2}{3}Kt) \frac{u_{t}}{u}\leq \left(\frac{n}{2t}+ \frac{nK}{4}\right)(1+\frac{2}{3}Kt)-\frac{nK}{12}.
\end{equation*}
Then, in   \cite{Jli}, Li and Xu derived gradient estimates for this positive solution to  $u_{t}=\Delta u$  as follows
\begin{equation*}
\frac{|\nabla u|^{2}}{u^{2}}-\Big(1+\frac{\sinh (Kt)\cosh(Kt)-Kt}{\sinh^{2}(Kt)}  \Big) \frac{u_{t}}{u}\leq \frac{nK}{2}(1+\coth (Kt)),
\end{equation*}
and its linearized version \cite{Dba},
\begin{equation*}
\frac{|\nabla u|^{2}}{u^{2}}-(1+\frac{2}{3}Kt) \frac{u_{t}}{u}\leq \frac{n}{2}(\frac{1}{t}+K+\frac{1}{3}K^{2}t).
\end{equation*}
In \cite{YAN}, Yang  generalized Ma's results \cite{MA} and proved local gradient estimates for positive solution  to the equation $u_{t}=\Delta u +au\log u+bu$ on complete noncompact manifolds with a fixed metric and curvature locally bounded below, where $a,b$ are real constants. Then Abolarinwa  et al. \cite{ABO} obtained elliptic gradient estimates for positive solutions to a weighted nonlinear parabolic equation $u_{t}=\Delta_{f}u-pu^{\beta}-qu$ on a weighted Riemannian manifold, where $\beta\in\mathbb{R}$ and $p$ and $q$ are smooth functions.
In \cite{ZZ, ZZ1} Zhang and Zhu derived the Li-Yau type gradient bounds for any positive solution $u$ to the heat equation $u_{t}=\Delta u$ under the integral Ricci curvature bounds introduced by Peterson and Wei \cite{P, WE},
\begin{equation}\label{icb}k(p,r)=\mathop{\sup}\limits_{x\in M}r^{2}\left(\fint_{B(x,r)}|Ric_{-}|^{p}d\mu\right)^{\frac{1}{p}}\leq \kappa
\end{equation}
as follows
$$\alpha\bar{J}\frac{\nabla u|^{2}}{u^{2}}-\frac{u_{t}}{u}\leq \frac{C_{1}}{\bar{J}}(1+\frac{C_{2}}{\bar{J}})+\frac{C_{3}}{t\bar{J}},$$
where $C_{i}=C_{i}(n,p,\alpha)$, $i=1,2,3$ and $\bar{J}=\bar{J}(t)$ is a decreasing exponential function.  Then, Rose \cite{RO} generalized the above inequality  to compact  Riemannian manifolds with negative parts of Ricci curvature in the Kato class. Also, Oliv\'{v} \cite{OL} generalized the results of   Zhang and Zhu \cite{ZZ, ZZ1} to   the positive solution to $u_{t}=\Delta u$ with Neumann boundary conditions  on a compact Riemannian submanifold with boundary $M^{n}\subset N$ satisfying the integral Ricci curvature assumption
$$D^{2}\mathop{\sup}\limits_{x\in N}r^{2}\left(\fint_{B(x,r)}|Ric_{-}|^{p}d\mu\right)^{\frac{1}{p}}\leq K$$
for $k(n,p)$ small enough, $p>\frac{n}{2}$, where ${\rm diam}(M)\leq D$.  Recently, Wang \cite{WW} derived the gradient estimates for any positive solution to the nonlinear parabolic equation $u_{t}=\Delta u+ au\log u$ under integral Curvature bounds.  \\
Motivated by the above works,
in this paper, we  will  derive some gradient estimates for  positive solutions of the  nonlinear parabolic equation
\begin{equation}\label{1}
u_{t}(x,t)=\Delta u(x,y)+au(x,t)(\log u(x,t))^{b}+q(x,t)A(u(x,t))
\end{equation}
on a complete Riemannian manifold $(M^{n},g)$ under  the integral Ricci curvature (\ref{icb}), where $q(x,t)$ is a function on of $C^{2}$, $A(u)$ is a function  of $C^{2}$ in $u$, and $a,b$ are real constants. \\
Let $f=\log u$, then
\begin{equation}\label{2}
(\Delta-\partial_{t})f=-|\nabla f|^{2}-af^{b}-qB(f),
\end{equation}
where $B(f)=\frac{A(u)}{u}$.
Then
 \begin{equation*}
B_{f}=A'(u)-\frac{A(u)}{u},\,\,\,\,\,B_{ff}=uA''(u)-A'(u)+\frac{A(u)}{u}.
\end{equation*}
Moreover, for $u>0$ we define several nonnegative real constants as follows:
\begin{eqnarray*}
&&\lambda_{1}:=\mathop{\sup}\limits_{B(O,\frac{1}{2})\times(0,\infty)}|B|,\qquad\qquad\quad\quad\,\,\theta_{1}:=\mathop{\sup}\limits_{B(O,\frac{1}{2})\times(0,\infty)}|q|,\\
&&\lambda_{2}:=\mathop{\sup}\limits_{B(O,\frac{1}{2})\times(0,\infty)}|B_{f}|,\qquad\qquad\quad\,\,\theta_{2}:=\mathop{\sup}\limits_{B(O,\frac{1}{2})\times(0,\infty)}|\nabla q|,\\
&&\lambda_{3}:=\mathop{\sup}\limits_{B(O,\frac{1}{2})\times(0,\infty)}|B_{ff}|,\qquad\qquad\quad\,\,\theta_{3}:=\mathop{\sup}\limits_{B(O,\frac{1}{2})\times(0,\infty)}|\Delta q|.
\end{eqnarray*}
Since we assume $u$   be  positive solution on compact Riemannian  manifold $M$, we can consider  ${D}_{1}\leq u(x,t)\leq {D}_{2}$  for some positive constants ${D}_{1}$ and ${D}_{2}$. For simplicity, we define  $\tilde{D}_{i}=\max\{|(\log \tilde{D}_{1})^{b-i}|, |(\log \tilde{D}_{2})^{b-i}|\}$ for $i=0,1,2$.\\
For prove of our results we need the following Lemmas. From \cite{ZZ} we have the following lemma.
\begin{lemma}\label{le}
There exists a $\kappa=\kappa(n,p)$ such that when $k(p,1)\leq \kappa$ for any $0<r\leq 1$, the system
\begin{equation}
\begin{cases}
\Delta J-2VJ-5\frac{|\nabla J|^{2}}{\epsilon J}-J_{t}=0,&\text{on}\quad B(O,r)\times(0,\infty),\\
J(.,0)=1,&\text{on}\quad B(O,r),\\
J(.,t)=1, &\text{on}\quad \partial B(O,r),
\end{cases}
\end{equation}
has  a unique solution for $t\in[0,\infty)$, which satisfies  $\bar{J}_{r}(t)\leq J(x,t)\leq1$, where
$$
\bar{J}_{r}(t)=2^{-\frac{1}{b-1}}\exp\{-2C\kappa r^{-2}(1+[2C(b-1)\kappa]^{\frac{n}{2p-n}})t\}
$$
for some constant $C=C(n,p)$ and $b=\frac{5}{\epsilon}$.
\end{lemma}
We have the following lemma from \cite{DWZ}.
\begin{lemma}\label{l2}
 Let $(M^{n},g)$ be a complete Riemannian manifold. For any $p>\frac{n}{2}$, there exists constants $k=k(p,n)$ and $C=C(p,n)$ such that if $k(p,1)\leq \kappa$, then for any geodesic ball $B(x,r)$ and $0<r\leq1$ there exists $\psi\in C_{0}^{\infty}((B(x,r))$ satisfying $0\leq \psi\leq1$, $\psi\equiv1$ in $B(x,\frac{r}{2})$, and $|\nabla \psi|^{2}+|\Delta \psi|\leq \frac{C}{r^{2}}$.
\end{lemma}
Similar to \cite{WW}, we introduce two $C^{1}$ functions $\alpha$ and $\phi:[0,\infty)\to (0,\infty)$. Let there exists two positive real numbers $\epsilon$ and $\lambda$ such that  these constants and  functions  $\alpha$ and $\phi$ satisfy the following conditions\\
\begin{itemize}
\item[(A1)] $\alpha(t)>1$,
\item[(A2)] $2(2-\epsilon)\frac{\phi}{n}\geq \left[2(2-\epsilon)\frac{\phi}{n}-\alpha'\right]\frac{1}{\alpha}$,
\item[(A3)] $2(2-\epsilon)\frac{\phi}{n}\bar{J}-\alpha'\geq0$,
\item[(A4)] $\alpha\phi'+(2-\epsilon)\frac{\phi{2}}{n}\bar{J}\geq0$,
\item[(A5)] $0<\frac{\alpha}{\alpha-1}\leq\lambda$,
\item[(A6)] $0<\delta \leq \frac{2}{1+n\lambda^{2}}$,
\end{itemize}
where $\alpha'=\frac{d}{dt}\alpha$ and $\phi'=\frac{d}{dt}\phi$.
Firstly, we give the following gradient estimates  for positive solutions to (\ref{1}) under integral curvature bound (\ref{icb}).
\begin{theorem}\label{t1}
Let $(M,g)$ be a $n$-dimensional complete Riemannain manifold. Suppose that $u$ be a positive  solution to (\ref{1}) on $M$.  Let there exists two positive real numbers $\epsilon$ and $\lambda$  and  two functions  $\alpha$ and $\phi$ satisfy  conditions $(A1)-(A6)$. For $p>\frac{n}{2}$, there exists  a constant $\kappa=\kappa(n,p)$ such that $k(p,1)\leq \kappa$, then for any point $O\in M$,  there exists positive constant $C_{1}$ such that
\begin{eqnarray}\nonumber
F&\leq& \frac{n\alpha^{2}}{2(2-\epsilon)  \bar{J}}\frac{1}{t}+ \frac{n\alpha^{2}}{(2-\epsilon)  \bar{J}}\left[\frac{C_{1}^{2}n\alpha^{2}}{2(2-\epsilon)  (\alpha-1)\bar{J}}+b\tilde{D}_{1}|a|+\theta_{1}\lambda_{2}+C_{1}\right]\\\label{f1}&&+\sqrt{ \frac{n\alpha^{2}}{(2-\epsilon)  \bar{J}}C_{2}}
\end{eqnarray} where $C_{1}=C(n,p)$ and
\begin{eqnarray}\nonumber
C_{2}&=&\frac{\left(|a|(\alpha-1)b\tilde{D}_{1}+5(\alpha-1)\theta_{1}\lambda_{2}+|a|\alpha b(b-1)\tilde{D}_{2}
+3(\alpha-1)\theta_{1}\lambda_{3}
\right)^{2}}{4\left[ \frac{(2-\epsilon)  J}{n\alpha^{2}}(\bar{J}-\alpha)^{2}-\epsilon\right]}\\&&
+
 \frac{1}{2\theta_{1}\lambda_{2}}\lambda_{1}\theta_{2}+|a|b\tilde{D}_{1}\alpha \phi
+\alpha \lambda_{1}\theta_{3}+\theta_{1}\lambda_{2}\alpha\phi+\frac{\alpha}{2\theta_{1}\lambda_{3}}\lambda_{2}\theta_{2}.
\end{eqnarray}
 The  two positive real numbers $\epsilon$ and $\lambda$  and  two functions  $\alpha$ and $\phi$ qre the following
\begin{itemize}
\item[1.] Li-Yau  type:
\begin{eqnarray*}
&&\alpha(t)= \text{constant}>1,\qquad \phi(t)=\frac{n}{(2-\epsilon)\alpha\bar{J}},\\
&&0<\epsilon\leq\frac{2(\alpha-1)^{2}}{(\alpha-1)^{2}+n\alpha^{2}}.
\end{eqnarray*}
\item[2.] Hamilton type:
\begin{eqnarray*}
&&\alpha(t)= \delta e^{\kappa t},\qquad \phi(t)=\frac{\delta \kappa n e^{\kappa t}}{2(2-\epsilon)\bar{J}},\\
&&0<\epsilon\leq\frac{2(\delta-1)^{2}}{(\delta-1)^{2}+n\delta^{2}}\quad\text{for any}\quad \delta>1.
\end{eqnarray*}
\item[3.] Li-Xu type:
\begin{eqnarray*}
&&\alpha(t)= \delta+\frac{\sinh (\kappa t)\cosh(\kappa t)-\kappa t}{\sinh(\kappa t)},\quad\phi(t)=\frac{n\kappa }{(2-\epsilon)\bar{J}}(\delta+\coth (\kappa t)),\\
&&0<\epsilon<\frac{\delta}{\delta-1}\quad\text{for any}\quad \delta>1.
\end{eqnarray*}
\item[4.] Linear Li-Xu type:
\begin{eqnarray*}
&&\alpha(t)= \delta+\kappa t,\qquad \phi(t)=\frac{n\kappa }{2(2-\epsilon)\bar{J}},\\
&&0<\epsilon\leq\frac{2(\delta-1)^{2}}{(\delta-1)^{2}+n\delta^{2}}\quad\text{for any}\quad \delta>1.
\end{eqnarray*}
\end{itemize}
\end{theorem}
As an application  of the global gradient estimates obtained in Corollary \ref{c1}, by integrating the gradient estimates in space-time we derive the following Harnack inequality. We first introduce the following notation.  Given $(y_{1},s_{1})\in M\times (0,T]$ and $(y_{2},s_{2})\in M\times (0,T]$ satisfying $s_{1}<s_{2}$, define
\begin{equation*}
\mathcal{J}(y_{1},s_{1},y_{2},s_{2})=\inf\int_{s_{1}}^{s_{2}}|{\zeta}'(t)|_{g(t)}^{2}dt,
\end{equation*}
and the infimum  is taken over the all smooth curves  $\zeta:[s_{1},s_{2}]\to M$ jointing $y_{1}$ and $y_{2}$.
\begin{corollary}\label{c2}
With the same assumptions in Theorem \ref{t1},  for $(y_{1},s_{1})\in  B(O,\frac{1}{2})\times (0,\infty)$ and $(y_{2},s_{2})\in B(O,\frac{1}{2})\times (0,\infty)$ such that $s_{1}<s_{2}$,    we have the following Harnack  inequalities.
\begin{itemize}
\item[1.]  For   Li-Yau type estimate,
 \begin{eqnarray}\nonumber
u(y_{1},s_{1})&\leq& u(y_{2},s_{2})(\frac{s_{2}}{s_{1}})^{ \frac{n\alpha^{2}}{2(2-\epsilon)  \bar{J}(T)}}\exp\left\{\frac{\alpha}{4\bar{J}(T)} \mathcal{J}(y_{1},s_{1},y_{2},s_{2}) \right.\\\label{c4}&&\left.+\left(C_{4}
+|a|\alpha D_{0}+\alpha \lambda_{1}\theta_{1}+\frac{n}{(2-\epsilon)\bar{J}(T)} \right)(s_{2}-s_{1})
\right\}.
\end{eqnarray}
Here,
\begin{eqnarray}\nonumber
C_{4}&=&  \frac{n\alpha^{2}}{(2-\epsilon)  \bar{J}(T)}\left[\frac{C_{1}^{2}n\alpha^{2}}{2(2-\epsilon)  (\alpha-1)\bar{J}(T)}+b\tilde{D}_{1}|a|+\theta_{1}\lambda_{2}+C_{1}\right]\\\nonumber&&+\sqrt{ \frac{n\alpha^{2}}{(2-\epsilon)  \bar{J}(T)}C_{2}}.
\end{eqnarray}
\item[2.]  For   Hamilton  type estimate,
\begin{eqnarray}\nonumber
u(y_{1},s_{1})&\leq& u(y_{2},s_{2})(\frac{s_{2}}{s_{1}})^{\frac{n\delta^{2} e^{2\kappa s_{2}}}{2(2-\epsilon)  \bar{J}(T)}}\exp\left\{\frac{\delta e^{\kappa s_{2}}}{4\bar{J}(T)}\mathcal{J}(y_{1},s_{1},y_{2},s_{2})\right.\\\label{c6}&&\left.
\left(C_{5}
+|a|\delta e^{\kappa s_{2}} D_{0}+\delta e^{\kappa s_{2}} \lambda_{1}\theta_{1}+\frac{\delta^{2}\kappa n e^{2\kappa s_{2}}}{2(2-\epsilon) \bar{J}(T)}\right)(s_{2}-s_{1})\right\},
\end{eqnarray}
where
\begin{eqnarray}\nonumber
C_{5}&=&  \frac{n\delta^{2} e^{2\kappa s_{2}}}{(2-\epsilon)  \bar{J}(T)}\left[\frac{C_{1}^{2}n\delta^{2} e^{2\kappa s_{2}}}{2(2-\epsilon)  (\delta e^{\kappa s_{1}}-1)\bar{J}(T)}+b\tilde{D}_{1}|a|+\theta_{1}\lambda_{2}+C_{1}\right]\\\nonumber&&+\sqrt{ \frac{n\delta^{2} e^{2\kappa s_{2}}}{(2-\epsilon)  \bar{J}(T)}C'_{5}},
\end{eqnarray}
and
\begin{eqnarray}\nonumber
C'_{5}&=&\frac{\left(|a|(\delta e^{\kappa s_{2}}-1)b\tilde{D}_{1}+5(\delta e^{\kappa s_{2}}-1)\theta_{1}\lambda_{2}+|a|\delta e^{\kappa s_{2}} b(b-1)\tilde{D}_{2}
+3(\delta e^{\kappa s_{2}}-1)\theta_{1}\lambda_{3}
\right)^{2}}{4\left[ \frac{(2-\epsilon)   \bar{J}}{n(\delta e^{\kappa s_{2}})^{2}}(\bar{J}-\delta e^{\kappa s_{1}})^{2}-\epsilon\right]}\\\nonumber&&
+
 \frac{1}{2\theta_{1}\lambda_{2}}\lambda_{1}\theta_{2}+|a|b\tilde{D}_{1}\delta e^{\kappa s_{2}} \phi
+\delta e^{\kappa s_{2}} \lambda_{1}\theta_{3}+\theta_{1}\lambda_{2}\delta e^{\kappa s_{2}}\phi+\frac{\delta e^{\kappa s_{2}}}{2\theta_{1}\lambda_{3}}\lambda_{2}\theta_{2}.
\end{eqnarray}
\item[3.] For  Li-Xu type estimate
\begin{eqnarray}\nonumber
u(y_{1},s_{1})&\leq& u(y_{2},s_{2})(\frac{s_{2}}{s_{1}})^{\frac{n\alpha^{2}(s_{2})}{2(2-\epsilon)  \bar{J}(T)}}\exp\left\{\frac{\alpha(s_{2})}{4\bar{J}(T)}
\mathcal{J}(y_{1},s_{1},y_{2},s_{2})
+\left(
C_{6}
+|a|\alpha(s_{2}) D_{0}\right.\right.\\\label{c8}&&\left.\left.+\alpha(s_{2}) \lambda_{1}\theta_{1}+\alpha(s_{2}) \frac{n\kappa}{(2-\epsilon)\bar{J}(T)}(\delta+\coth(\kappa s_{1}))
\right)
(s_{2}-s_{1})\right\},
\end{eqnarray}
where
\begin{eqnarray}\nonumber
C_{6}&=&  \frac{n\alpha^{2}(s_{2})}{(2-\epsilon)  \bar{J}(T)}\left[\frac{C_{1}^{2}n\alpha^{2}(s_{2})}{2(2-\epsilon)  (\alpha(s_{1})-1)\bar{J}(T)}+b\tilde{D}_{1}|a|+\theta_{1}\lambda_{2}+C_{1}\right]\\\nonumber&&+\sqrt{ \frac{n\alpha^{2}(s_{2})}{(2-\epsilon)  \bar{J}(T)}C'_{6}},
\end{eqnarray}
and
\begin{eqnarray}\nonumber
C'_{6}&=&\frac{\left(|a|(\alpha(s_{2})-1)b\tilde{D}_{1}+5(\alpha(s_{2})-1)\theta_{1}\lambda_{2}+|a|\alpha(s_{2}) b(b-1)\tilde{D}_{2}
+3(\alpha(s_{2})-1)\theta_{1}\lambda_{3}
\right)^{2}}{4\left[ \frac{(2-\epsilon)   \bar{J}}{n\alpha^{2}(s_{2})}(\bar{J}-\alpha(s_{1}))^{2}-\epsilon\right]}\\\nonumber&&
+
 \frac{1}{2\theta_{1}\lambda_{2}}\lambda_{1}\theta_{2}+|a|b\tilde{D}_{1}\alpha(s_{2}) \phi
+\alpha(s_{2}) \lambda_{1}\theta_{3}+\theta_{1}\lambda_{2}\alpha(s_{2})\phi+\frac{\alpha(s_{2})}{2\theta_{1}\lambda_{3}}\lambda_{2}\theta_{2}.
\end{eqnarray}
where $\alpha(s_{2})=\left(\delta+\frac{\sinh (\kappa s_{2})\cosh( \kappa s_{2}) -\kappa s_{2}}{\sinh(\kappa s_{2})}\right)$.
\item[4.] For  linear Li-Xu type gradient estimate,
\begin{eqnarray}\nonumber
u(y_{1},s_{1})&\leq& u(y_{2},s_{2})(\frac{s_{2}}{s_{1}})^{ \frac{n(\delta+\kappa s_{2})^{2}}{2(2-\epsilon)  \bar{J}(T)}}\exp\left\{\frac{\delta+\kappa s_{2}}{4\bar{J}(T)}
\mathcal{J}(y_{1},s_{1},y_{2},s_{2})
+\left(
C_{7}
+|a|\delta D_{0}\right.\right.\\\label{c10}&&\left.\left.+\delta \lambda_{1}\theta_{1}+\frac{n\kappa \delta}{(2-\epsilon)  \bar{J}(T)}
\right)
(s_{2}-s_{1})\right.\\\nonumber&&\left.+\frac{\kappa}{2}\left(|a|D_{0}+\lambda_{1}\theta_{1}+ \frac{n\kappa}{(2-\epsilon)  \bar{J}(T)}  \right)(s_{2}^{2}-s_{1}^{2})\right\},
\end{eqnarray}
\begin{eqnarray}\nonumber
C_{7}&=&  \frac{n(\delta+\kappa s_{2})^{2}}{(2-\epsilon)  \bar{J}(T)}\left[\frac{C_{1}^{2}n(\delta+\kappa s_{2})^{2}}{2(2-\epsilon)  (\delta+\kappa s_{1}-1)\bar{J}(T)}+b\tilde{D}_{1}|a|+\theta_{1}\lambda_{2}+C_{1}\right]\\\nonumber&&+\sqrt{ \frac{n(\delta+\kappa s_{2})^{2}}{(2-\epsilon)  \bar{J}(T)}C'_{7}}
\end{eqnarray}
and
\begin{eqnarray}\nonumber
C'_{7}&=&\frac{\left((\delta+\kappa s_{2}-1)\left( |a|b\tilde{D}_{1}+5\theta_{1}\lambda_{2}+3\theta_{1}\lambda_{3}\right)+|a|(\delta+\kappa s_{2}) b(b-1)\tilde{D}_{2}
\right)^{2}}{4\left[ \frac{(2-\epsilon)   \bar{J}}{n(\delta+\kappa s_{2})^{2}}(\bar{J}-\delta-\kappa s_{1})^{2}-\epsilon\right]}\\\nonumber&&
+
 \frac{1}{2\theta_{1}\lambda_{2}}\lambda_{1}\theta_{2}+|a|b\tilde{D}_{1}(\delta+\kappa s_{2}) \phi
+(\delta+\kappa s_{2}) \lambda_{1}\theta_{3}+\theta_{1}\lambda_{2}(\delta+\kappa s_{2})\phi+\frac{\delta+\kappa s_{2}}{2\theta_{1}\lambda_{3}}\lambda_{2}\theta_{2}.
\end{eqnarray}
\end{itemize}
\end{corollary}
\section{proofs of  results}
For prove our result, we need the following lemma.
\begin{lemma}Let $(M^{n},g)$ be a complete Riemannian manifold satisfies the hypotheses of Theorem \ref{t1}.
Let $J=J(x,t)$ be a positive smooth function, and $$F=J|\nabla f|^{2}-\alpha f_{t}+a\alpha f^{b}+\alpha q B-\alpha \phi$$
where $f_{t}=\partial_{t} f$. Then
\begin{eqnarray}\nonumber
(\Delta-\partial_{t})F&\geq&\frac{(2-\epsilon)  J}{n\alpha^{2}}F^{2}+\frac{2(2-\epsilon)  (\alpha-J)J}{n\alpha^{2}}F|\nabla f|^{2}\\\nonumber&&+\frac{(2-\epsilon)  J}{n\alpha^{2}}(J-\alpha)^{2}|\nabla f|^{4}
-2 \langle \nabla f,\nabla F \rangle
\\\nonumber&&-\epsilon J|\nabla f|^{4}
+2a(\alpha-J)bf^{b-1}|\nabla f|^{2}+2(\alpha-J)\langle \nabla f,\nabla(qB)\rangle\\\label{14}&&+a\alpha \Delta f^{b}+\alpha B\Delta q+\alpha q\Delta B+2\alpha \langle \nabla q,\nabla B\rangle.
\end{eqnarray}
\end{lemma}
\begin{proof}
By the Bochner-Weitzenbb\"{o}ck formula
$$\frac{1}{2}\Delta |\nabla f|^{2}=|{\rm Hess}f|^{2}+\langle \nabla f,\nabla\Delta f\rangle+Ric(\nabla f,\nabla f)$$
 we have
\begin{eqnarray}\nonumber
\Delta F&=&2 J|{\rm Hess}f|^{2}+2J\langle \nabla f,\nabla\Delta f\rangle+2JRic(\nabla f,\nabla f)+ |\nabla f|^{2}\Delta J\\\nonumber&&+2\langle \nabla J,\nabla |\nabla f|^{2}\rangle -\alpha \Delta f_{t}+a\alpha \Delta f^{b}+\alpha B\Delta q+\alpha q\Delta B\\\label{3}&&+2\alpha \langle \nabla q,\nabla B\rangle.
\end{eqnarray}
Applying  (\ref{2}) in (\ref{3}) we get
\begin{eqnarray}\nonumber
\Delta F&=&2 J|{\rm Hess}f|^{2}+2J\langle \nabla f,\nabla\left(f_{t}-|\nabla f|^{2}-af^{b}-qB \right)\rangle\\\nonumber&&+2JRic(\nabla f,\nabla f)+ |\nabla f|^{2}\Delta J+2\langle \nabla J,\nabla |\nabla f|^{2}\rangle -\alpha \Delta f_{t}\\\label{4}&&+a\alpha \Delta f^{b}+\alpha B\Delta q+\alpha q\Delta B+2\alpha \langle \nabla q,\nabla B\rangle.
\end{eqnarray}
Also, we have
\begin{eqnarray}\nonumber
F_{t}&=&J_{t}|\nabla f|^{2}+2J\langle \nabla f,\nabla f_{t}\rangle-\alpha f_{tt}-\alpha' f_{t}+a\alpha' f^{b}+ab\alpha f^{b-1}f_{t}\\\label{5}&&
+\alpha'qB+\alpha q_{t} B+\alpha q B_{t}-\alpha' \phi-\alpha \phi'\\\nonumber&=&
J_{t}|\nabla f|^{2}+2J\langle \nabla f,\nabla f_{t}\rangle-\alpha \Delta f_{t}-2\alpha \langle \nabla f,\nabla f_{t}\rangle\\\nonumber&&-\alpha' f_{t}+a\alpha' f^{b}
+\alpha'qB-\alpha' \phi-\alpha \phi'.
\end{eqnarray}
Combining (\ref{4}) and (\ref{5}), we infer
\begin{eqnarray}\nonumber
(\Delta-\partial_{t})F&=&2 J|{\rm Hess}f|^{2}+2JRic(\nabla f,\nabla f)\\\nonumber&&
-2 \langle \nabla f,J\nabla |\nabla f|^{2}+|\nabla f|^{2}\nabla J-\alpha \nabla f_{t}+a\alpha\nabla f^{b}+\alpha \nabla(qB)-\nabla(\alpha \phi) \rangle\\\nonumber&&
+2a(\alpha-J)bf^{b-1}|\nabla f|^{2}+2(\alpha-J)\langle \nabla f,\nabla(qB)\rangle+2|\nabla f|^{2}\langle \nabla J, \nabla f\rangle\\\nonumber&&+ |\nabla f|^{2}\Delta J+a\alpha \Delta f^{b}+\alpha B\Delta q+\alpha q\Delta B+2\alpha \langle \nabla q,\nabla B\rangle\\\nonumber&&+2\langle \nabla J,\nabla |\nabla f|^{2}\rangle -J_{t}|\nabla f|^{2}+\alpha' f_{t}-a\alpha' f^{b}
-\alpha'qB+\alpha' \phi+\alpha \phi'\\\nonumber&=&
2 J|{\rm Hess}f|^{2}+2JRic(\nabla f,\nabla f)
-2 \langle \nabla f,\nabla F \rangle\\\nonumber&&
+2a(\alpha-J)bf^{b-1}|\nabla f|^{2}+2(\alpha-J)\langle \nabla f,\nabla(qB)\rangle+2|\nabla f|^{2}\langle \nabla J, \nabla f\rangle\\\nonumber&&+ |\nabla f|^{2}\Delta J+a\alpha \Delta f^{b}+\alpha B\Delta q+\alpha q\Delta B+2\alpha \langle \nabla q,\nabla B\rangle\\\label{6}&&+2\langle \nabla J,\nabla |\nabla f|^{2}\rangle -J_{t}|\nabla f|^{2}+\alpha' f_{t}-a\alpha' f^{b}
-\alpha'qB+\alpha' \phi+\alpha \phi'.
\end{eqnarray}
For any real number $\epsilon>0$, by Cauchy inequality we have
\begin{equation}\label{7}
2\langle \nabla J,\nabla |\nabla f|^{2}\rangle\geq -\epsilon J|{\rm Hess}f|^{2}-\frac{4|\nabla J|^{2}}{\epsilon J}|\nabla f|^{2}
\end{equation}
and
\begin{equation}\label{8}
2|\nabla f|^{2}\langle \nabla J, \nabla f\rangle\geq -\epsilon J|\nabla f|^{4}-\frac{|\nabla J|^{2}}{\epsilon J}|\nabla f|^{2}.
\end{equation}
Substituting (\ref{7}) and (\ref{8}) into (\ref{6}), we obtain
\begin{eqnarray}\nonumber
(\Delta-\partial_{t})F&\geq&
(2-\epsilon)  J|{\rm Hess}f|^{2}
-2 \langle \nabla f,\nabla F \rangle\\\nonumber&&
+\left(  \Delta J-\frac{5|\nabla J|^{2}}{\epsilon J}-J_{t}-2VJ\right)|\nabla f|^{2}-\epsilon J|\nabla f|^{4}\\\nonumber&&
+2a(\alpha-J)bf^{b-1}|\nabla f|^{2}+2(\alpha-J)\langle \nabla f,\nabla(qB)\rangle\\\nonumber&&+a\alpha \Delta f^{b}+\alpha B\Delta q+\alpha q\Delta B+2\alpha \langle \nabla q,\nabla B\rangle\\\label{9}&& +\alpha' f_{t}-a\alpha' f^{b}
-\alpha'qB+\alpha' \phi+\alpha \phi',
\end{eqnarray}
where $V=|Ric_{-}|$.
From lemma \ref{le} we deduce that
\begin{eqnarray}\nonumber
(\Delta-\partial_{t})F&\geq&
(2-\epsilon)  J|{\rm Hess}f+\frac{\phi}{n}g|^{2}-2(2-\epsilon)J\frac{\phi}{n}\Delta f-(2-\epsilon)J\frac{\phi^{2}}{n}\\\nonumber&&
-2 \langle \nabla f,\nabla F \rangle
-\epsilon J|\nabla f|^{4}\\\nonumber&&
+2a(\alpha-J)bf^{b-1}|\nabla f|^{2}+2(\alpha-J)\langle \nabla f,\nabla(qB)\rangle\\\nonumber&&+a\alpha \Delta f^{b}+\alpha B\Delta q+\alpha q\Delta B+2\alpha \langle \nabla q,\nabla B\rangle\\\label{10}&& +\alpha' f_{t}-a\alpha' f^{b}
-\alpha'qB+\alpha' \phi+\alpha \phi'.
\end{eqnarray}
Plugging (\ref{2}) into (\ref{10}), we infer
\begin{eqnarray}\nonumber
(\Delta-\partial_{t})F&\geq&
(2-\epsilon)  J|{\rm Hess}f+\frac{\phi}{n}g|^{2}+2(2-\epsilon)J\frac{\phi}{n}|\nabla f|^{2}\\\nonumber&&-\left[2(2-\epsilon)\frac{\phi}{n}J-\alpha'\right]f_{t}+\left[2(2-\epsilon)\frac{\phi}{n}J-\alpha'\right](af^{b}+qB)\\\nonumber&&-\left[2(2-\epsilon)\frac{\phi}{n}J-\alpha'\right]\phi+(2-\epsilon)J\frac{\phi^{2}}{n}
-2 \langle \nabla f,\nabla F \rangle
-\epsilon J|\nabla f|^{4}\\\nonumber&&
+2a(\alpha-J)bf^{b-1}|\nabla f|^{2}+2(\alpha-J)\langle \nabla f,\nabla(qB)\rangle\\\label{10}&&+a\alpha \Delta f^{b}+\alpha B\Delta q+\alpha q\Delta B+2\alpha \langle \nabla q,\nabla B\rangle
+\alpha \phi'.
\end{eqnarray}
By conditions  $(A2)-(A4)$ we get
\begin{equation}
\begin{cases}
2(2-\epsilon)\frac{\phi}{n}\leq\left[2(2-\epsilon)\frac{\phi}{n}J-\alpha'\right]\frac{1}{\alpha},\\
2(2-\epsilon)\frac{\phi}{n}J-\alpha'>0,\\
\alpha \phi'+(2-\epsilon)J\frac{\phi^{2}}{n}.
\end{cases}
\end{equation}
Then
\begin{eqnarray}\nonumber
(\Delta-\partial_{t})F&\geq&
(2-\epsilon)  J|{\rm Hess}f+\frac{\phi}{n}g|^{2}\\\nonumber&&+\left[2(2-\epsilon)\frac{\phi}{n}J-\alpha'\right]\frac{F}{\alpha}
-2 \langle \nabla f,\nabla F \rangle
-\epsilon J|\nabla f|^{4}\\\nonumber&&
+2a(\alpha-J)bf^{b-1}|\nabla f|^{2}+2(\alpha-J)\langle \nabla f,\nabla(qB)\rangle\\\label{11}&&+a\alpha \Delta f^{b}+\alpha B\Delta q+\alpha q\Delta B+2\alpha \langle \nabla q,\nabla B\rangle.
\end{eqnarray}
On the other hand, Cauchy inequality yields
\begin{eqnarray}\nonumber
|{\rm Hess}f+\frac{\phi}{n}g|^{2}&\geq& \frac{1}{n}(\Delta f+\phi)^{2}=\frac{1}{n}(f_{t}-|\nabla f^{2}-af^{b}-qB+\phi)^{2}\\\label{12}&=&
\frac{1}{n\alpha^{2}}\left[ F+(\alpha-J)|\nabla f|^{2}\right]^{2}.
\end{eqnarray}
Hence
\begin{eqnarray}\nonumber
(\Delta-\partial_{t})F&\geq&\frac{(2-\epsilon)  J}{n\alpha^{2}}F^{2}+\frac{2(2-\epsilon)  (\alpha-J)J}{n\alpha^{2}}F|\nabla f|^{2}\\\nonumber&&+\frac{(2-\epsilon)  J}{n\alpha^{2}}(J-\alpha)^{2}|\nabla f|^{4}
+\left[2(2-\epsilon)\frac{\phi}{n}J-\alpha'\right]\frac{F}{\alpha}
-2 \langle \nabla f,\nabla F \rangle
\\\nonumber&&-\epsilon J|\nabla f|^{4}
+2a(\alpha-J)bf^{b-1}|\nabla f|^{2}+2(\alpha-J)\langle \nabla f,\nabla(qB)\rangle\\\label{13}&&+a\alpha \Delta f^{b}+\alpha B\Delta q+\alpha q\Delta B+2\alpha \langle \nabla q,\nabla B\rangle.
\end{eqnarray}
Therefore (\ref{13}) yields (\ref{14}).
\end{proof}
\begin{proof}[Proof of Theorem \ref{t1}]
According to Lemma \ref{l2} we choose  a cut-off function $\psi$  such that \begin{equation}\label{cf}\begin{cases}
0\leq \psi\leq1 \\supp \psi \subset\subset B(O,\frac{1}{2}), \\\psi\equiv1\quad \text{in}\quad B(O,\frac{1}{2}),\\ |\nabla \psi|^{2}+|\Delta \psi|\leq C(p,n)=C.
\end{cases}
\end{equation}
Multiplying both sides of inequality (\ref{14}) with $t\psi^{4}$, we conclude
\begin{eqnarray}\nonumber
&&\sqrt{t}\psi^{2}(\Delta-\partial_{t})(\sqrt{t}\psi^{2}F)\\\nonumber&=&t\psi^{4} (\Delta-\partial_{t})F+\sqrt{t}\psi^{2}\left(
\sqrt{t}F\Delta \psi^{2}+\sqrt{t}\langle\nabla F,\nabla\psi^{2}\rangle- \frac{1}{2\sqrt{t}}\psi^{2}F
\right)\\\nonumber
&\geq&\frac{(2-\epsilon)  J}{n\alpha^{2}}(\sqrt{t}\psi^{2}F)^{2}+\frac{2(2-\epsilon)  (\alpha-J)J}{n\alpha^{2}}t\psi^{4}F|\nabla f|^{2}\\\nonumber&&+\frac{(2-\epsilon)  J}{n\alpha^{2}}(J-\alpha)^{2}t\psi^{4}|\nabla f|^{4}
-2t\psi^{4} \langle \nabla f,\nabla F \rangle
\\\nonumber&&-\epsilon t\psi^{4}  J|\nabla f|^{4}
+2at\psi^{4}(\alpha-J)bf^{b-1}|\nabla f|^{2}+2t\psi^{4}(\alpha-J)\langle \nabla f,\nabla(qB)\rangle\\\label{15}&&+at\psi^{4}\alpha \Delta f^{b}+\alpha t\psi^{4} B\Delta q+t\psi^{4}\alpha q\Delta B+2t\psi^{4}\alpha \langle \nabla q,\nabla B\rangle\\\nonumber&&+
\psi^{2}\left(
{t}F\Delta \psi^{2}+{t}\langle\nabla F,\nabla\psi^{2}\rangle- \frac{1}{2}\psi^{2}F
\right).
\end{eqnarray}
   Since
   $$\alpha \Delta f=\alpha ( f_{t}-|\nabla f|^{2}-af^{b}-qB)=-F-(\alpha-J)|\nabla f|^{2}-\alpha \phi$$
    we have
    \begin{eqnarray}\nonumber
 \alpha \Delta f^{b}&=&\alpha\left(b(b-1)f^{b-2}|\nabla f|^{2}+bf^{b-1}\Delta f\right)\\\label{16}&=&
 \alpha b(b-1)f^{b-2}|\nabla f|^{2}-bf^{b-1}F-bf^{b-1}(\alpha-J)|\nabla f|^{2}-bf^{b-1}\alpha \phi,
    \end{eqnarray}
    and
        \begin{eqnarray}\nonumber
 \alpha \Delta B&=&\alpha\left( B_{ff}|\nabla f|^{2}+B_{f}\Delta f\right)\\\nonumber&=&
 \alpha B_{ff}|\nabla f|^{2}+B_{f}\left(-F-(\alpha-J)|\nabla f|^{2}-\alpha \phi\right)
 \\\label{17}&\geq&
  -\alpha\lambda_{3}|\nabla f|^{2}+\lambda_{2}\left(-F-(\alpha-J)|\nabla f|^{2}-\alpha \phi\right).
    \end{eqnarray}
   For any $\delta_{1}>0$  and  $\delta_{2}>0$ we have
    \begin{equation}\label{18}
     \langle \nabla q,\nabla B\rangle= B_{f}\langle \nabla q,\nabla f\rangle\geq - \frac{1}{4\delta_{1}}\lambda_{2}\theta_{2} -\delta_{1} |\nabla f|^{2}
    \end{equation}
and
        \begin{equation}\label{19}
     \langle \nabla f,\nabla (qB)\rangle= qB_{f}|\nabla f|^{2}+B\langle \nabla f,\nabla q\rangle\geq-\theta_{1}\lambda_{2}|\nabla f|^{2} - \frac{1}{4\delta_{2}}\lambda_{1}\theta_{2} -\delta_{2} |\nabla f|^{2}.
    \end{equation}
    Applying (\ref{16})-(\ref{19}) into (\ref{15}), we get
    \begin{eqnarray}\nonumber
&&\sqrt{t}\psi^{2}(\Delta-\partial_{t})(\sqrt{t}\psi^{2}F)\\\nonumber
&\geq&\frac{(2-\epsilon)  J}{n\alpha^{2}}(\sqrt{t}\psi^{2}F)^{2}+\frac{2(2-\epsilon)  (\alpha-J)J}{n\alpha^{2}}t\psi^{4}F|\nabla f|^{2}-2t\psi^{4} \langle \nabla f,\nabla F \rangle\\\nonumber&&+\left[ \frac{(2-\epsilon)  J}{n\alpha^{2}}(J-\alpha)^{2}-\epsilon\right]t\psi^{4}|\nabla f|^{4}
+2at\psi^{4}(\alpha-J)bf^{b-1}|\nabla f|^{2}\\\nonumber&&+2t\psi^{4}(\alpha-J)\left( -\theta_{1}\lambda_{2}|\nabla f|^{2} - \frac{1}{4\delta_{1}}\lambda_{1}\theta_{2} -\delta_{2} |\nabla f|^{2}\right)\\\nonumber&&
\\\label{20}&&
+at\psi^{4}\left(
 \alpha b(b-1)f^{b-2}|\nabla f|^{2}-bf^{b-1}F-bf^{b-1}(\alpha-J)|\nabla f|^{2}-bf^{b-1}\alpha \phi
\right)\\\nonumber&&-\alpha t\psi^{4} \lambda_{1}\theta_{3}+t\psi^{4}\theta_{1} \left(   -\alpha\lambda_{3}|\nabla f|^{2}+\lambda_{2}\left(-F-(\alpha-J)|\nabla f|^{2}-\alpha \phi\right)\right)\\\nonumber&&+2t\psi^{4}\alpha \left(  - \frac{1}{4\delta_{2}}\lambda_{2}\theta_{2} -\delta_{2} |\nabla f|^{2}\right)+
\psi^{2}\left(
{t}F\Delta \psi^{2}+{t}\langle\nabla F,\nabla\psi^{2}\rangle- \frac{1}{2}\psi^{2}F
\right).
\end{eqnarray}
     For any $T>0$,  suppose that $(x_{1},t_{1})$ is a maximum point of $\sqrt{t}\psi^{2} F$ in $B(O,1)\times(0,T]$. We can assume that  the value is positive, because otherwise   the proof is trivial. At this point, since $\nabla(\psi^{2}F)=0$, we have $\psi\nabla F=-2F\nabla \psi$. Therefore,
         \begin{eqnarray}\nonumber
&&\sqrt{t}\psi^{2}(\Delta-\partial_{t})(\sqrt{t}\psi^{2}F)\\\nonumber
&\geq&\frac{(2-\epsilon)  J}{n\alpha^{2}}(\sqrt{t}\psi^{2}F)^{2}+\frac{2(2-\epsilon)  (\alpha-J)J}{n\alpha^{2}}t\psi^{4}F|\nabla f|^{2}+4t\psi^{4} F\langle \nabla f,\nabla \psi \rangle\\\nonumber&&+\left[ \frac{(2-\epsilon)  J}{n\alpha^{2}}(J-\alpha)^{2}-\epsilon\right]t\psi^{4}|\nabla f|^{4}
+2at\psi^{4}(\alpha-J)bf^{b-1}|\nabla f|^{2}\\\nonumber&&+2t\psi^{4}(\alpha-J)\left( -\theta_{1}\lambda_{2}|\nabla f|^{2} - \frac{1}{4\delta_{1}}\lambda_{1}\theta_{2} -\delta_{1} |\nabla f|^{2}\right)\\\nonumber&&
\\\label{21}&&
+at\psi^{4}\left(
 \alpha b(b-1)f^{b-2}|\nabla f|^{2}-bf^{b-1}F-bf^{b-1}(\alpha-J)|\nabla f|^{2}-bf^{b-1}\alpha \phi
\right)\\\nonumber&&-\alpha t\psi^{4} \lambda_{1}\theta_{3}+t\psi^{4}\theta_{1} \left(   -\alpha\lambda_{3}|\nabla f|^{2}+\lambda_{2}\left(-F-(\alpha-J)|\nabla f|^{2}-\alpha \phi\right)\right)\\\nonumber&&+2t\psi^{4}\alpha \left(  - \frac{1}{4\delta_{2}}\lambda_{2}\theta_{2} -\delta_{2} |\nabla f|^{2}\right)+
2t\psi^{3}
F\Delta \psi-6{t}\psi^{2}F |\nabla \psi|^{2}- \frac{1}{2}\psi^{4}F.
\end{eqnarray}
    Choosing $0<\epsilon< \frac{2}{1+n\lambda^{2}}$ and  by using $\frac{\alpha}{\alpha-1}\leq \lambda $, we have $\frac{2-\epsilon}{n\alpha^{2}}(1-\alpha)^{2}-\epsilon>0$. Since $J\leq1$, we conclude $$\frac{(2-\epsilon)  J}{n\alpha^{2}}(J-\alpha)^{2}-\epsilon>0.$$
 Applying (\ref{cf}) in (\ref{21}), there exists a constant $C_{1}$ such that
        \begin{eqnarray}\nonumber
&&\sqrt{t}\psi^{2}(\Delta-\partial_{t})(\sqrt{t}\psi^{2}F)\\\nonumber
&\geq&\frac{(2-\epsilon)  J}{n\alpha^{2}}(\sqrt{t}\psi^{2}F)^{2}+\frac{2(2-\epsilon)  (\alpha-J)J}{n\alpha^{2}}t\psi^{4}F|\nabla f|^{2}-C_{1}t\psi^{4} F| \nabla f| \\\nonumber&&+\left[ \frac{(2-\epsilon)  J}{n\alpha^{2}}(J-\alpha)^{2}-\epsilon\right]t\psi^{4}|\nabla f|^{4}
+2at\psi^{4}(\alpha-J)bf^{b-1}|\nabla f|^{2}\\\nonumber&&+2t\psi^{4}(\alpha-J)\left( -\theta_{1}\lambda_{2}|\nabla f|^{2} - \frac{1}{4\delta_{1}}\lambda_{1}\theta_{2} -\delta_{1} |\nabla f|^{2}\right)\\\nonumber&&
\\\label{22}&&
+at\psi^{4}\left(
 \alpha b(b-1)f^{b-2}|\nabla f|^{2}-bf^{b-1}F-bf^{b-1}(\alpha-J)|\nabla f|^{2}-bf^{b-1}\alpha \phi
\right)\\\nonumber&&-\alpha t\psi^{4} \lambda_{1}\theta_{3}+t\psi^{4}\theta_{1} \left(   -\alpha\lambda_{3}|\nabla f|^{2}+\lambda_{2}\left(-F-(\alpha-J)|\nabla f|^{2}-\alpha \phi\right)\right)\\\nonumber&&+2t\psi^{4}\alpha \left(  - \frac{1}{4\delta_{2}}\lambda_{2}\theta_{2} -\delta_{2} |\nabla f|^{2}\right)-C_{1}{t}\psi^{2}F - \frac{1}{2}\psi^{4}F
.
\end{eqnarray}
We consider $\delta_{1}=\theta_{1}\lambda_{2}$ and $\delta_{2}=\theta_{1}\lambda_{3}$, then   we have
        \begin{eqnarray}\nonumber
&&\sqrt{t}\psi^{2}(\Delta-\partial_{t})(\sqrt{t}\psi^{2}F)\\\nonumber
&\geq&\frac{(2-\epsilon)  J}{n\alpha^{2}}(\sqrt{t}\psi^{2}F)^{2}+\frac{2(2-\epsilon)  (\alpha-J)J}{n\alpha^{2}}t\psi^{4}F|\nabla f|^{2}-C_{1}t\psi^{4} F| \nabla f| \\\nonumber&&+\left[ \frac{(2-\epsilon)  J}{n\alpha^{2}}(J-\alpha)^{2}-\epsilon\right]t\psi^{4}|\nabla f|^{4}
\\\nonumber&&-
t\psi^{4}\left(|a|(\alpha-J)b\tilde{D}_{1}+5(\alpha-J)\theta_{1}\lambda_{2}+|a|\alpha b(b-1)\tilde{D}_{2}
+3(\alpha-J)\theta_{1}\lambda_{3}
\right)|\nabla f|^{2}
\\\label{23}&&
+\left(-b\tilde{D}_{1}|a|t\psi^{4}-\theta_{1}\lambda_{2}t\psi^{4}
-C_{1}{t}\psi^{2}- \frac{1}{2}\psi^{4}
\right)F\\\nonumber&&-
 \frac{t\psi^{4}}{2\theta_{1}\lambda_{2}}\lambda_{1}\theta_{2}-|a|t\psi^{4}b\tilde{D}_{1}\alpha \phi
-\alpha t\psi^{4} \lambda_{1}\theta_{3}-t\psi^{4}\theta_{1}\lambda_{2}\alpha\phi-\frac{t\psi^{4}\alpha}{2\theta_{1}\lambda_{3}}\lambda_{2}\theta_{2}
.
\end{eqnarray}
Using inequality $Ax^{2}-Bx\geq -\frac{B^{2}}{4A}$, $(A,B>0)$, we have
\begin{eqnarray}\nonumber
\frac{2(2-\epsilon)  (\alpha-J)J}{n\alpha^{2}}t\psi^{4}F|\nabla f|^{2}-C_{1}t\psi^{4} F| \nabla f| &\geq &-\frac{C_{1}^{2}n\alpha^{2}}{2(2-\epsilon)  (\alpha-J){J}}t\psi^{4} F\\
&\geq &-\frac{C_{1}^{2}n\alpha^{2}}{2(2-\epsilon)  (\alpha-1)\bar{J}}t\psi^{4} F
\end{eqnarray}
and
\begin{eqnarray}\nonumber
&&\left[ \frac{(2-\epsilon)  J}{n\alpha^{2}}(J-\alpha)^{2}-\epsilon\right]t\psi^{4}|\nabla f|^{4}
\\\nonumber&&-
t\psi^{4}\left(|a|(\alpha-J)b\tilde{D}_{1}+5(\alpha-J)\theta_{1}\lambda_{2}+|a|\alpha b(b-1)\tilde{D}_{2}
+3(\alpha-J)\theta_{1}\lambda_{3}
\right)|\nabla f|^{2}\\\nonumber&&
\geq -\frac{\left(|a|(\alpha-J)b\tilde{D}_{1}+5(\alpha-J)\theta_{1}\lambda_{2}+|a|\alpha b(b-1)\tilde{D}_{2}
+3(\alpha-J)\theta_{1}\lambda_{3}
\right)^{2}}{4\left[ \frac{(2-\epsilon)  \bar{J}}{n\alpha^{2}}( \bar{J}-\alpha)^{2}-\epsilon\right]}t\psi^{4}.
\end{eqnarray}
Hence
        \begin{eqnarray}\nonumber
&&\sqrt{t}\psi^{2}(\Delta-\partial_{t})(\sqrt{t}\psi^{2}F)\\\nonumber
&\geq&\frac{(2-\epsilon)  J}{n\alpha^{2}}(\sqrt{t}\psi^{2}F)^{2}-\left[\frac{C_{1}^{2}n\alpha^{2}}{2(2-\epsilon)  (\alpha-1)\bar{J}}+b\tilde{D}_{1}|a|+\theta_{1}\lambda_{2}\right]t\psi^{4} F \\\nonumber&& -\frac{\left(|a|(\alpha-J)b\tilde{D}_{1}+5(\alpha-J)\theta_{1}\lambda_{2}+|a|\alpha b(b-1)\tilde{D}_{2}
+3(\alpha-J)\theta_{1}\lambda_{3}
\right)^{2}}{4\left[ \frac{(2-\epsilon)  \bar{J}}{n\alpha^{2}}( \bar{J}-\alpha)^{2}-\epsilon\right]}t\psi^{4}
\\\label{23}&&
-\left(
C_{1}{t}\psi^{2}+ \frac{1}{2}\psi^{4}
\right)F\\\nonumber&&-
 \frac{t\psi^{4}}{2\theta_{1}\lambda_{2}}\lambda_{1}\theta_{2}-|a|t\psi^{4}b\tilde{D}_{1}\alpha \phi
-\alpha t\psi^{4} \lambda_{1}\theta_{3}-t\psi^{4}\theta_{1}\lambda_{2}\alpha\phi-\frac{t\psi^{4}\alpha}{2\theta_{1}\lambda_{3}}\lambda_{2}\theta_{2}
.
\end{eqnarray}
By max principle, we infer
        \begin{eqnarray}\nonumber
0&\geq&\frac{(2-\epsilon)  J}{n\alpha^{2}}(\sqrt{t}\psi^{2}F)^{2}-\left[\frac{C_{1}^{2}n\alpha^{2}}{2(2-\epsilon)  (\alpha-1)\bar{J}}+b\tilde{D}_{1}|a|+\theta_{1}\lambda_{2}\right]t\psi^{4} F \\\nonumber&& -\frac{\left(|a|(\alpha-J)b\tilde{D}_{1}+5(\alpha-J)\theta_{1}\lambda_{2}+|a|\alpha b(b-1)\tilde{D}_{2}
+3(\alpha-J)\theta_{1}\lambda_{3}
\right)^{2}}{4\left[ \frac{(2-\epsilon)   \bar{J}}{n\alpha^{2}}( \bar{J}-\alpha)^{2}-\epsilon\right]}t\psi^{4}
\\\label{23}&&
-\left(
C_{1}{t}\psi^{2}+ \frac{1}{2}\psi^{4}
\right)F\\\nonumber&&-
 \frac{t\psi^{4}}{2\theta_{1}\lambda_{2}}\lambda_{1}\theta_{2}-|a|t\psi^{4}b\tilde{D}_{1}\alpha \phi
-\alpha t\psi^{4} \lambda_{1}\theta_{3}-t\psi^{4}\theta_{1}\lambda_{2}\alpha\phi-\frac{t\psi^{4}\alpha}{2\theta_{1}\lambda_{3}}\lambda_{2}\theta_{2}
.
\end{eqnarray}
So
\begin{eqnarray}\nonumber
F&\leq& \frac{n\alpha^{2}}{2(2-\epsilon)  \bar{J}}\frac{1}{t}+ \frac{n\alpha^{2}}{(2-\epsilon)  \bar{J}}\left[\frac{C_{1}^{2}n\alpha^{2}}{2(2-\epsilon)  (\alpha-1)\bar{J}}+b\tilde{D}_{1}|a|+\theta_{1}\lambda_{2}+C_{1}\right]\\&&+\sqrt{ \frac{n\alpha^{2}}{(2-\epsilon)  \bar{J}}C_{2}}
\end{eqnarray} where
\begin{eqnarray}\nonumber
C_{2}&=&\frac{\left(|a|(\alpha-1)b\tilde{D}_{1}+5(\alpha-1)\theta_{1}\lambda_{2}+|a|\alpha b(b-1)\tilde{D}_{2}
+3(\alpha-1)\theta_{1}\lambda_{3}
\right)^{2}}{4\left[ \frac{(2-\epsilon)   \bar{J}}{n\alpha^{2}}(\bar{J}-\alpha)^{2}-\epsilon\right]}\\&&
+
 \frac{1}{2\theta_{1}\lambda_{2}}\lambda_{1}\theta_{2}+|a|b\tilde{D}_{1}\alpha \phi
+\alpha \lambda_{1}\theta_{3}+\theta_{1}\lambda_{2}\alpha\phi+\frac{\alpha}{2\theta_{1}\lambda_{3}}\lambda_{2}\theta_{2}.
\end{eqnarray}
To obtain the required result on $F(x,t)$ fo an appropriate range of $x\in M$, we get $\psi(x,t)=1$ whenever $d(x,O)<1$ and since $(x_{1},t_{1})$ is the maximum point of $\sqrt{t}\psi^{2}F$ in $B(O,1)\times(0,T]$, we get 
$$F(x,T)=\frac{\sqrt{T}\psi^{2}(x,T)F(x,T)}{\sqrt{T}\psi^{2}(x,T)}\leq \frac{\sqrt{t_{1}}\psi^{2}(x_{1},t_{1})F(x_{1},t_{1})}{\sqrt{T}\psi^{2}(x,T)}\leq F(x_{1},t_{1}).$$
Since $\bar{J}(t)$ is a decreasing function and $T$ is arbitrary, then inequality (\ref{f1}) holds.
 Note that
 \begin{itemize}
\item[(1 )] Li-Yau  type gradient estimate\\
Let
\begin{eqnarray*}
&&\alpha(t)= \text{constant}>1,\qquad \phi(t)=\frac{n}{(2-\epsilon)\alpha\bar{J}},\\
&&0<\epsilon\leq\frac{2(\alpha-1)^{2}}{(\alpha-1)^{2}+n\alpha^{2}}.
\end{eqnarray*}
By a direct calculation, we have
 \begin{equation*}
2(2-\epsilon)\frac{\phi}{n}\bar{J}-\alpha'=\frac{2}{\alpha}\geq0,
\end{equation*}
and
 \begin{equation*}
\alpha\phi'+(2-\epsilon)\frac{\phi{2}}{n}\bar{J}=\frac{n}{(2-\epsilon)\alpha^{2}\bar{J}}\geq0.
\end{equation*}
Also, $\alpha(t)>1$    and  $2(2-\epsilon)\frac{\phi}{n}\geq \left[2(2-\epsilon)\frac{\phi}{n}-\alpha'\right]\frac{1}{\alpha}$. Hence   $\alpha$ and $\phi$ in this case satisfy the conditions (A1)-(A6).
\item[(2)] Hamilton type gradient estimate\\
Let
\begin{eqnarray*}
&&\alpha(t)= \delta e^{\kappa t},\qquad \phi(t)=\frac{\delta \kappa  n e^{\kappa t}}{2(2-\epsilon)\bar{J}},\\
&&0<\epsilon\leq\frac{2(\delta-1)^{2}}{(\delta-1)^{2}+n\delta^{2}}\quad\text{for any}\quad \delta>1.
\end{eqnarray*}
Direct calculation shows
 \begin{equation*}
2(2-\epsilon)\frac{\phi}{n}\bar{J}-\alpha'=\delta \kappa e^{\kappa t}-\delta \kappa e^{\kappa t}=0
\end{equation*}
and
 \begin{equation*}
\alpha\phi'+(2-\epsilon)\frac{\phi{2}}{n}\bar{J}=\frac{3\delta^{2} \kappa^{2} n e^{2\kappa t}}{4(2-\epsilon)\bar{J}}\geq0.
\end{equation*}
We see that  $\alpha(t)>1$   and $2(2-\epsilon)\frac{\phi}{n}\geq \left[2(2-\epsilon)\frac{\phi}{n}-\alpha'\right]\frac{1}{\alpha}$. So, the functions   $\alpha$  and   $\phi$ in this case satisfy the conditions (A1)-(A6).
\item[(3)] Li-Xu type gradient estimate\\
Let
\begin{eqnarray*}
&&\alpha(t)= \delta+\frac{\sinh (\kappa t)\cosh(\kappa t)-\kappa t}{\sinh(\kappa t)},\quad\phi(t)=\frac{n\kappa }{(2-\epsilon)\bar{J}}(\delta+\coth (\kappa t)),\\
&&0<\epsilon<\frac{\delta}{\delta-1}\quad\text{for any}\quad \delta>1.
\end{eqnarray*}
We  have $\alpha'(t)\leq0$ then
 \begin{equation*}
 2(2-\epsilon)\frac{\phi}{n}\bar{J}-\alpha'\geq0
\end{equation*}
and
 \begin{eqnarray*}
\alpha\phi'+(2-\epsilon)\frac{\phi{2}}{n}\bar{J}\geq 0.
\end{eqnarray*}
 Thus, the functions   $\alpha$ and   $\phi$ in this case satisfy the conditions (A1)-(A6).
\item[(4)] Linear Li-Xu type gradient estimate\\
Let
\begin{eqnarray*}
&&\alpha(t)= \delta+\kappa t,\qquad \phi(t)=\frac{n\kappa }{2(2-\epsilon)\bar{J}},\\
&&0<\epsilon\leq\frac{2(\delta-1)^{2}}{(\delta-1)^{2}+n\delta^{2}}\quad\text{for any}\quad \delta>1.
\end{eqnarray*}
Then we get
 \begin{equation*}
2(2-\epsilon)\frac{\phi}{n}\bar{J}-\alpha'=0
\end{equation*}
and
 \begin{equation*}
\alpha\phi'+(2-\epsilon)\frac{\phi{2}}{n}\bar{J}=\frac{n \kappa^{2}}{4(2-\epsilon)\bar{J}}\geq0.
\end{equation*}
Therefore, the functions   $\alpha$ and   $\phi$  in this case satisfy the conditions (A1)-(A6).
\end{itemize}
\end{proof}
\begin{proof}[Proof of Corollary \ref{c2}]
Let  $\zeta(t)$ be a shortest  the geodesic joining  $y_{1}$ and  $y_{2}$ with $\zeta(s_{1})=y_{1}$ and $\zeta(s_{2})=y_{2}$. Now consider the path $(\zeta(t),t)$ in space-time. From Theorem  \ref{t1}, we have the following gradient estimate
\begin{eqnarray}\nonumber
J|\nabla \log u|^{2}-\alpha ( \log u)_{t}+a\alpha( \log u)^{b}+\alpha q B-\alpha \phi \leq \frac{n\alpha^{2}}{2(2-\epsilon)  \bar{J}}\frac{1}{t}+C_{3}
\end{eqnarray}
where
\begin{eqnarray}\nonumber
C_{3}&=&  \frac{n\alpha^{2}}{(2-\epsilon)  \bar{J}}\left[\frac{C_{1}^{2}n\alpha^{2}}{2(2-\epsilon)  (\alpha-1)\bar{J}}+b\tilde{D}_{1}|a|+\theta_{1}\lambda_{2}+C_{1}\right]+\sqrt{ \frac{n\alpha^{2}}{(2-\epsilon)  \bar{J}}C_{2}}.
\end{eqnarray}
Integrating this inequality along $\zeta$, we get
\begin{eqnarray}\nonumber
&&\log\frac{u(y_{1},s_{1})}{u(y_{2},s_{2})}\\\nonumber
&&=-\int_{s_{1}}^{s_{2}}\frac{d}{dt}\big(\log u(\zeta(t),t) \big)dt\\\nonumber
&&=-\int_{s_{1}}^{s_{2}}\Big(\partial_{t}(\log u)+\langle \nabla (\log u)(\zeta(t),t) ,\dot{\zeta}(t)\rangle \Big)dt\\\nonumber&&
\leq \int_{s_{1}}^{s_{2}}\left\{ -\frac{\bar{J}(t)}{\alpha} |\nabla (\log u)|^{2}+\frac{n\alpha^{2}}{2(2-\epsilon)  \bar{J}(t)}\frac{1}{t}+C_{3}
-a\alpha( \log u)^{b}-\alpha q B+\alpha \phi  \right.\\\nonumber&&\left.-\langle \nabla (\log u) ,\dot{\zeta}(t)\rangle\right\}dt\\\label{c1}&&
\leq\int_{s_{1}}^{s_{2}}\left\{ \frac{\alpha|\dot{\zeta}(t)|^{2}}{4\bar{J}(t)}+\frac{n\alpha^{2}}{2(2-\epsilon)  \bar{J}(t)}\frac{1}{t}+C_{3}
+|a|\alpha D_{0}+\alpha \lambda_{1}\theta_{1}+\alpha \phi \right\}dt,
\end{eqnarray}
where in the computation above  we have used  (\ref{e1c2}) to obtain the inequality in the third line and used inequality $-\tilde{a}x^{2}-\tilde{b}x\leq \frac{{\tilde b}^{2}}{4\tilde{a}}$ to arrive at   the inequality in the fourth line. Since $\bar{J}(t)$ is a decreasing function, we conclude
$$\max \frac{1}{\bar{J}(t)}\leq  \frac{1}{\bar{J}(T)}.$$
Then, inequality (\ref{c1}) becomes
\begin{eqnarray}\nonumber
\log\frac{u(y_{1},s_{1})}{u(y_{2},s_{2})}
&\leq&\frac{1}{4\bar{J}(T)}\int_{s_{1}}^{s_{2}}\alpha|\dot{\zeta}(t)|^{2} dt\\\label{c2}&&+
\int_{s_{1}}^{s_{2}}\left\{ \frac{n\alpha^{2}}{2(2-\epsilon)  \bar{J}(T)}\frac{1}{t}+C_{4}
+|a|\alpha D_{0}+\alpha \lambda_{1}\theta_{1}+\alpha \phi \right\}dt.
\end{eqnarray}
Here
\begin{eqnarray}\nonumber
C_{4}&=&  \frac{n\alpha^{2}}{(2-\epsilon)  \bar{J}(T)}\left[\frac{C_{1}^{2}n\alpha^{2}}{2(2-\epsilon)  (\alpha-1)\bar{J}(T)}+b\tilde{D}_{1}|a|+\theta_{1}\lambda_{2}+C_{1}\right]\\&&+\sqrt{ \frac{n\alpha^{2}}{(2-\epsilon)  \bar{J}(T)}C_{2}}.
\end{eqnarray}
Now in Li-Yau type estimate we have
\begin{eqnarray}\nonumber
\log\frac{u(y_{1},s_{1})}{u(y_{2},s_{2})}
&\leq&\frac{\alpha}{4\bar{J}(T)}\int_{s_{1}}^{s_{2}}|\dot{\zeta}(t)|^{2} dt+
 \frac{n\alpha^{2}}{2(2-\epsilon)  \bar{J}(T)}\log (\frac{s_{2}}{s_{1}})\\\label{c3}&&+\left(C_{4}
+|a|\alpha D_{0}+\alpha \lambda_{1}\theta_{1}+\frac{n}{(2-\epsilon)\bar{J}(T)} \right)(s_{2}-s_{1}).
\end{eqnarray}
This implies that inequality (\ref{c4}). \\
For Hamilton type estimate we have, $\delta e^{\kappa s_{1}}\leq\alpha(t)=\delta e^{\kappa t}\leq \delta e^{\kappa s_{2}}$. Then
\begin{eqnarray}\nonumber
\log\frac{u(y_{1},s_{1})}{u(y_{2},s_{2})}
&\leq&\frac{\delta e^{\kappa s_{2}}}{4\bar{J}(T)}\int_{s_{1}}^{s_{2}}|\dot{\zeta}(t)|^{2} dt
+\int_{s_{1}}^{s_{2}}\left\{ \frac{n\delta^{2} e^{2\kappa s_{2}}}{2(2-\epsilon)  \bar{J}(T)}\frac{1}{t}+C_{5}\right.\\\label{c5}&&\left.
+|a|\delta e^{\kappa s_{2}} D_{0}+\delta e^{\kappa s_{2}} \lambda_{1}\theta_{1}+\frac{\delta^{2}\kappa n e^{2\kappa s_{2}}}{2(2-\epsilon) \bar{J}(T)}\right\}dt,
\end{eqnarray}
where
\begin{eqnarray}\nonumber
C_{5}&=&  \frac{n\delta^{2} e^{2\kappa s_{2}}}{(2-\epsilon)  \bar{J}(T)}\left[\frac{C_{1}^{2}n\delta^{2} e^{2\kappa s_{2}}}{2(2-\epsilon)  (\delta e^{\kappa s_{1}}-1)\bar{J}(T)}+b\tilde{D}_{1}|a|+\theta_{1}\lambda_{2}+C_{1}\right]\\\nonumber&&+\sqrt{ \frac{n\delta^{2} e^{2\kappa s_{2}}}{(2-\epsilon)  \bar{J}(T)}C'_{5}}.
\end{eqnarray}
The inequality (\ref{c5}) implies that inequality (\ref{c6}). \\
In the Li-Xu type estimate we have $\alpha'(t)\geq0$, then $\alpha(t)\leq \alpha(s_{2})=\left(\delta+\frac{\sinh (\kappa s_{2})\cosh( \kappa s_{2}) -\kappa s_{2}}{\sinh(\kappa s_{2})}\right)$. Also, $\phi(t)\leq \frac{n\kappa}{(2-\epsilon)\bar{J}(T)}(\delta+\coth(\kappa s_{1}))$. Therefore, we obtain
\begin{eqnarray}\nonumber
\log\frac{u(y_{1},s_{1})}{u(y_{2},s_{2})}
&\leq&\frac{\alpha(s_{2})}{4\bar{J}(T)}\int_{s_{1}}^{s_{2}}|\dot{\zeta}(t)|^{2} dt\\\label{c7}&&+
\int_{s_{1}}^{s_{2}}\left\{ \frac{n\alpha^{2}(s_{2})}{2(2-\epsilon)  \bar{J}(T)}\frac{1}{t}+C_{6}
+|a|\alpha(s_{2}) D_{0}+\alpha(s_{2}) \lambda_{1}\theta_{1}\right.\\\nonumber&&\left.+\alpha(s_{2}) \frac{n\kappa}{(2-\epsilon)\bar{J}(T)}(\delta+\coth(\kappa s_{1})) \right\}dt,
\end{eqnarray}
where
\begin{eqnarray}\nonumber
C_{6}&=&  \frac{n\alpha^{2}(s_{2})}{(2-\epsilon)  \bar{J}(T)}\left[\frac{C_{1}^{2}n\alpha^{2}(s_{2})}{2(2-\epsilon)  (\alpha(s_{1})-1)\bar{J}(T)}+b\tilde{D}_{1}|a|+\theta_{1}\lambda_{2}+C_{1}\right]\\\nonumber&&+\sqrt{ \frac{n\alpha^{2}(s_{2})}{(2-\epsilon)  \bar{J}(T)}C'_{6}}.
\end{eqnarray}
Hence, we get (\ref{c8}).\\
In the linear Li-Xu type gradient estimate, we have
 \begin{eqnarray}\nonumber
\log\frac{u(y_{1},s_{1})}{u(y_{2},s_{2})}
&\leq&\frac{\delta+\kappa s_{2}}{4\bar{J}(T)}\int_{s_{1}}^{s_{2}}|\dot{\zeta}(t)|^{2} dt+
\int_{s_{1}}^{s_{2}}\left\{ \frac{n(\delta+\kappa s_{2})^{2}}{2(2-\epsilon)  \bar{J}(T)}\frac{1}{t}+C_{7}\right.\\\label{c9}&&\left.
+|a|(\delta+\kappa t) D_{0}+(\delta+\kappa t) \lambda_{1}\theta_{1}+\frac{n\kappa (\delta+\kappa t)}{(2-\epsilon)  \bar{J}(T)}  \right\}dt.
\end{eqnarray}
Here,
\begin{eqnarray}\nonumber
C_{7}&=&  \frac{n(\delta+\kappa s_{2})^{2}}{(2-\epsilon)  \bar{J}(T)}\left[\frac{C_{1}^{2}n(\delta+\kappa s_{2})^{2}}{2(2-\epsilon)  (\delta+\kappa s_{1}-1)\bar{J}(T)}+b\tilde{D}_{1}|a|+\theta_{1}\lambda_{2}+C_{1}\right]\\\nonumber&&+\sqrt{ \frac{n(\delta+\kappa s_{2})^{2}}{(2-\epsilon)  \bar{J}(T)}C'_{7}}.
\end{eqnarray}
Thus we infer (\ref{c10}).
\end{proof}

\end{document}